\documentclass{amsart}
\usepackage{amssymb,stmaryrd,cmap}
\usepackage[inline]{enumitem}

\newtheorem{theorem}{Theorem}[section]
\newtheorem{claim}{Claim}[theorem]

\newtheorem{lemma}[theorem]{Lemma}
\newtheorem{fact}[theorem]{Fact}	

\newtheorem{cor}[theorem]{Corollary}
\newtheorem{mainthm}{Theorem}

\newtheorem{maincor}[mainthm]{Corollary}

\theoremstyle{definition}
\newtheorem{defn}[theorem]{Definition}

\theoremstyle{remark}
\newtheorem{remark}[theorem]{Remark}

\DeclareMathOperator{\ubd}{{\sf unbounded}}
\DeclareMathOperator{\seq}{Seq}

\DeclareMathOperator{\otp}{otp}

\DeclareMathOperator{\acc}{acc}

\DeclareMathOperator{\ad}{AD}

\DeclareMathOperator{\cf}{cf}
\DeclareMathOperator{\dom}{dom}
\DeclareMathOperator{\im}{Im}
\DeclareMathOperator{\cov}{cov}
\DeclareMathOperator{\cof}{cof}

\newcommand\s{\subseteq}
\newcommand*\axiomfont[1]{\textsf{\textup{#1}}}
\newcommand\ma{\axiomfont{MA}}
\def\br{\blacktriangleright}
\newcommand\bd{\textup{bd}}
\def\dl{D\ell}

\newcommand{\stick}{{\ensuremath \mspace{2mu}\mid\mspace{-12mu} {\raise0.6em\hbox{$\bullet$}}}}

\date{Preprint as of September 21, 2022. For the latest version, visit \textsf{http://p.assafrinot.com/54}.}

\author{Assaf Rinot}
\address{Department of Mathematics, Bar-Ilan University, Ramat-Gan 5290002, Israel.}
\urladdr{http://www.assafrinot.com}

\author{Roy Shalev}
\address{Department of Mathematics, Bar-Ilan University, Ramat-Gan 5290002, Israel.}
\urladdr{https://roy-shalev.github.io/}

\author{Stevo Todorcevic}
\address{Department of Mathematics, University of Toronto, Toronto, Canada, M5S 2E4. Institut de Math\'ematiques de Jussieu, UMR 7586, 2 pl. Jussieu, Case 7012, 75251 Paris Cedex 05, France. Mathematical Institute SANU, Kneza Mihaila 36, 11001 Belgrade, Serbia.}
\email{stevo@math.toronto.edu}
\email{stevo.todorcevic@imj-prg.fr}
\email{stevo@mi.sanu.ac.rs}

\title{A new small Dowker space}
\subjclass[2010]{Primary 03E05, 54G20; Secondary 03E65}
\begin{document}
\begin{abstract} It is proved that if there exists a Luzin set, or if either $\stick$ or $\diamondsuit(\mathfrak b)$ hold,
then a strong instance of the guessing principle $\clubsuit_{\ad}$ holds at the first uncountable cardinal.
In particular, any of the above hypotheses entails the existence of a Dowker space of size $\aleph_1$.
\end{abstract}
\maketitle
 
\section{Introduction}
A \emph{Dowker space} is a normal topological space $\mathbb X$ whose product with the unit
interval $\mathbb X\times[0,1]$ is not normal. Whether such a space exists was asked by Dowker back in 1951 \cite{Dowker_C.H.}.
As of now, there are just three constructions of Dowker spaces in ZFC: Rudin's space of size $(\aleph_\omega)^{\aleph_0}$ \cite{Rudin_first_Dowker_ZFC_example}, Balogh's space of size continuum \cite{Balogh_space}, and the Kojman-Shelah space of size $\aleph_{\omega+1}$ \cite{MR1605988}.
Rudin's \emph{Conjecture~4} from \cite{MR1078646},
asserting that there exists a Dowker space of size $\aleph_1$ remains open.

\medskip

In \cite{paper48}, the guessing principle $\clubsuit_{\ad}$ was introduced,
and it was shown that 
for every regular uncountable cardinal $\kappa$, each of the following two implies the existence of a Dowker space of size $\kappa$:
\begin{enumerate}
\item[(i)] $\clubsuit_{\ad}(\mathcal S,1,2)$ holds
for a partition $\mathcal S$ of some nonreflecting stationary subset of $\kappa$
into infinitely many stationary sets;
\item[(ii)] $\clubsuit_{\ad}(\{E^{\kappa}_{\lambda}\},\lambda,1)$ holds,
where $\kappa$ is the successor of a regular cardinal $\lambda$.
\end{enumerate}

It was also shown that in each of the scenarios of \cite{kappa_Dowker_from_souslin_Rudin,de_Caux_space,MR628595,Good_Dowker_large_cardinals} in which there exists a Dowker space of size $\kappa$,
either (i) or (ii) indeed hold.

The exact definition of the guessing principle under discussion reads as follows:

\begin{defn}\label{clubad}  Let $\mathcal S$ be a collection of stationary subsets of a regular uncountable cardinal $\kappa$,
and $\mu,\theta$ be nonzero cardinals $<\kappa$.
The principle 
$\clubsuit_{\ad}(\mathcal S,\mu,\theta)$ asserts the existence of a sequence $\langle \mathcal A_\alpha\mid \alpha\in\bigcup\mathcal S\rangle$ such that:
\begin{enumerate}
\item For every $\alpha\in\acc(\kappa)\cap\bigcup\mathcal S$, $\mathcal A_\alpha$ is a pairwise disjoint family of $\mu$ many cofinal subsets of $\alpha$;
\item For every $\mathcal B\s[\kappa]^\kappa$ of size $\theta$, for every $S\in\mathcal S$, there are stationarily many $\alpha\in S$ such that $\sup(A\cap B)=\alpha$ for all $A\in\mathcal A_\alpha$ and $B\in\mathcal B$;
\item For all $A\neq A'$ from $\bigcup_{S\in\mathcal S}\bigcup_{\alpha\in S}\mathcal A_\alpha$, $\sup(A\cap A')<\sup(A)$.
\end{enumerate} 
\end{defn}
\begin{remark} 
The variation $\clubsuit_{\ad}(\mathcal S,\mu,{<}\theta)$ is defined in the obvious way.
\end{remark}

In \cite{two_more_S_spaces}, Juh\'asz, Kunen and Rudin constructed
a Dowker space of size $\aleph_1$ assuming the continuum hypothesis.
This was then improved by the third author \cite[p.~53]{TodPartTop} who got such a space from the existence of a Luzin set (cf.~\cite{MR1283587}).

The first main result of this paper shows that also in the above scenario, (i) and (ii)  hold.
This answers Question~2.35 of \cite{paper48} in the negative.
 
\begin{mainthm}\label{thma} Let $\lambda=\lambda^{<\lambda}$ be an infinite regular cardinal.

Then $(1)\implies(2)\implies(3)$:
\begin{enumerate}
\item There exists a $\lambda^+$-Luzin subset of ${}^{\lambda}\lambda$;
\item There exists a tight strongly unbounded coloring $c:\lambda\times\lambda^+\rightarrow\lambda$;
\item For every partition $\mathcal S$ of $E^{\lambda^+}_{\lambda}$ into stationary sets,
$\clubsuit_{\ad}(\mathcal S,\lambda,{<}\lambda)$ holds.
\end{enumerate}
\end{mainthm}

Note that as in the case $\lambda:=\aleph_0$, for every infinite cardinal $\lambda=\lambda^{<\lambda}$,
the existence of a $\lambda^+$-Luzin subset of ${}^{\lambda}\lambda$ follows from $2^\lambda=\lambda^+$.
Our second main result derives $\clubsuit_{\ad}$ from another consequence of $2^\lambda=\lambda^+$, namely,
from the \emph{stick} principle:

\begin{mainthm}\label{thmb} Let $\lambda$ be an infinite regular cardinal.

Then $\stick(\lambda^+)$ implies that
for every partition $\mathcal S$ of $E^{\lambda^+}_{\lambda}$ into stationary sets,
$\clubsuit_{\ad}(\mathcal S,\lambda,\lambda)$ holds.
\end{mainthm} 

What's interesting about Theorem~B is that it uncovers 
a scenario for the existence of a Dowker space that was not known before.
In particular, it yields the following contribution to the small Dowker space problem:

\begin{maincor} $\stick$ entails the existence of a Dowker spaces of size $\aleph_1$. In fact, it entails the existence of $2^{\aleph_1}$ many pairwise nonhomeomorphic such spaces.
\end{maincor}

\subsection{Organization of this paper}
In Section~\ref{Sec2}, we define \emph{unbounded}, \emph{strongly unbounded} and \emph{tight} colorings, and provide sufficient conditions for their existence. The proof of the implication $(1)\implies(2)$ of Theorem~\ref{thma} will be found there.

In Section~\ref{Sec3}, we get an instance of $\clubsuit_{\ad}(\ldots)$ from a tight strongly unbounded coloring. In particular, the proof of the implication $(2)\implies(3)$ of Theorem~\ref{thma} will be found there.

In Section~\ref{Sec4}, we get instances of $\clubsuit_{\ad}(\ldots)$ from the principles $\stick(\lambda^+)$ and $\diamondsuit(\mathfrak b)$.
In particular, the proof of Theorem~\ref{thmb} will be found there.

In the Appendix, we slightly extend Clause~(i) above, showing that for every regular uncountable cardinal $\kappa$, if $\clubsuit_{\ad}(\mathcal S,1,2)$ holds
for an infinite partition $\mathcal S$ of some nonreflecting stationary subset of $\kappa$
into $\mu$ many stationary sets,
then there are $2^\mu$ many pairwise nonhomeomorphic Dowker spaces of size $\kappa$.

\section{Unbounded colorings and generalized Luzin sets}\label{Sec2}
In this section, $\kappa$ denotes a regular uncountable cardinal, and $\nu$ and $\lambda$ denote infinite cardinals $<\kappa$.
\begin{defn}\label{deftrees} Suppose that $c:\nu\times\kappa\rightarrow\lambda$ is a coloring. 
\begin{itemize}
\item For each $\beta<\kappa$, derive the fiber map $c_\beta:\nu\rightarrow\lambda$ via $c_\beta(\eta):=c(\eta,\beta)$;
\item $c$ is \emph{unbounded} iff for every cofinal $B\s\kappa$, there is an $\eta<\nu$ such that 
$$\sup\{ c_\beta(\eta)\mid \beta\in B\}=\lambda;$$
\item $c$ is \emph{strongly unbounded} iff for every cofinal $B\s\kappa$, there are an $\eta<\nu$ 
and a map $t:\eta\rightarrow\lambda$ such that 
$$\sup\{ c_\beta(\eta)\mid \beta\in B\  \&\ t\s c_\beta\}=\lambda;$$
\item For every $\mathbb T\s{}^{<\nu}\lambda$, let $[\mathbb T]_c:=\{ \beta<\kappa\mid \forall\eta<\nu\,(c_\beta\restriction\eta\in\mathbb  T)\}$;
\item Set $\mathcal T_c:=\{ \mathbb T\s{}^{<\nu}\lambda\mid \sup([\mathbb T]_c)=\kappa\}$;
\item $c$ is \emph{tight} iff $\cf(\mathcal T_c,{\supseteq})\le\kappa$.
\end{itemize}
\end{defn}
\begin{remark}
It is clear that ${}^{<\nu}\lambda\in\mathcal T_c$, so that $1\le|\mathcal T_c|\le2^{(\lambda^{<\nu})}$.
\end{remark}

\subsection{From unbounded to strongly unbounded}
\begin{lemma}\label{lemma23} Suppose that $\lambda=\aleph_0$ or $\lambda$ is strongly inaccessible.

Then any unbounded coloring $c:\lambda\times\kappa\rightarrow\lambda$ 
is strongly unbounded.
\end{lemma}
\begin{proof} Suppose that $c:\lambda\times\kappa\rightarrow\lambda$  is a given unbounded coloring. Let $B$
be some cofinal subset of $\kappa$.
By hypothesis, there exists an $\eta<\lambda$ such that 
$\sup\{ c_\beta(\eta)\mid \beta\in B\}=\lambda$.
For the least such $\eta$, it follows that there exists some ordinal $\mu<\lambda$ such that
$$\{  c_\beta(i)\mid \beta\in B, i<\eta\}\s \mu.$$
As $|{}^\eta\mu|<\cf(\lambda)=\lambda$, there must exist some $t\in{}^\eta\mu$ such that 
$$\sup\{ c_\beta(\eta)\mid \beta\in B\ \& \ t\s c_\beta\}=\lambda,$$
as sought.
\end{proof}

\begin{defn}[Shelah, {\cite[\S2]{Sh:107}}] For a regular uncountable cardinal $\lambda$:
\begin{enumerate}
\item $\dl_\lambda$ asserts the existence of a sequence $\langle \mathcal P_\eta\mid\eta<\lambda\rangle$ such that:
\begin{itemize}
\item  for every $\eta<\lambda$, $\mathcal P_\eta\s\mathcal P(\eta)$ and $|\mathcal P_\eta|<\lambda$;
\item for every  $A\s\lambda$, for stationarily many $\eta<\lambda$,
$A\cap\eta\in\mathcal P_\eta$. 
\end{itemize}
\item $\dl^*_\lambda$ asserts the existence of a sequence $\langle \mathcal P_\eta\mid\eta<\lambda\rangle$ such that:
\begin{itemize}
\item  for every $\eta<\lambda$, $\mathcal P_\eta\s\mathcal P(\eta)$ and $|\mathcal P_\eta|<\lambda$;
\item for every  $A\s\lambda$, for club many $\eta<\lambda$,
$A\cap\eta\in\mathcal P_\eta$. 
\end{itemize}
\end{enumerate}
\end{defn}

\begin{fact}[Shelah, {\cite[Claim 3.2]{sh:460} and \cite[Claim 2.5]{sh:922}}]\label{dlfacts} For a regular uncountable cardinal $\lambda$:
\begin{enumerate}
\item If $\lambda$ is strongly inaccessible, then $\dl^*_\lambda$ holds;
\item $\diamondsuit_\lambda$ implies $\dl_\lambda$,
and $\diamondsuit^*_\lambda$ implies $\dl^*_\lambda$;
\item If $\lambda\ge\beth_\omega$ then $\dl_\lambda$ iff $\lambda^{<\lambda}=\lambda$;
\item If $\lambda$ is a successor of an uncountable cardinal, then $\dl_\lambda$ iff $\lambda^{<\lambda}=\lambda$.
\end{enumerate}
\end{fact}

\begin{lemma}\label{lemma26} Suppose that $\lambda$ is a regular uncountable cardinal and $\dl^*_\lambda$ holds.
Then there exists a strongly unbounded coloring $c:\lambda\times\kappa\rightarrow\lambda$,
for $\kappa:=\mathfrak b_\lambda$.
\end{lemma}
\begin{proof} We commence with verifying the following variation of weak diamond.
\begin{claim} For every function $F:{}^{<\lambda}\lambda\rightarrow\lambda$,
there exists a function $g:\lambda\rightarrow\lambda$ with the property that for every function $f:\lambda\rightarrow\lambda$,
the following set covers a club:
$$\{\eta<\lambda\mid F(f\restriction\eta)\le g(\eta)\}.$$
\end{claim}
\begin{proof}
Using $\dl^*_\lambda$, we may fix a sequence $\langle \mathcal F_\eta\mid\eta<\lambda\rangle$ such that:
\begin{itemize}
\item  for every $\eta<\lambda$, $\mathcal F_\eta\s{}^\eta\eta$ and $|\mathcal F_\eta|<\lambda$;
\item for every function $f:\lambda\rightarrow\lambda$, for club many $\eta<\lambda$,
$f\restriction\eta\in\mathcal F_\eta$. 
\end{itemize}
Now, given any function  $F:{}^{<\lambda}\lambda\rightarrow\lambda$,
define an oracle function $g:\lambda\rightarrow\lambda$ via
$$g(\eta):=\sup\{ F(t)\mid t\in\mathcal F_\eta\}.$$
Next, given any function $f:\lambda\rightarrow\lambda$, the set $C:=\{ \eta<\lambda\mid f\restriction\eta\in\mathcal F_\eta\}$ covers a club,
and it is clear that, for every $\eta\in C$, $g(\eta)\ge F( f\restriction\eta)$.
\end{proof}

For functions $f,g\in{}^\lambda\lambda$, let $f<_{cl}g$ iff $\{\alpha<\lambda\mid f(\alpha)<g(\alpha)\}$ covers a club.
By \cite[Theorem~6]{CuSh:541}, $\mathfrak b_\lambda$ coincides with the least size of a family of functions from $\lambda$ to $\lambda$
that is not bounded with respect to $<_{cl}$. It follows that we may construct a $<_{cl}$-increasing sequence of functions $\langle f_\beta\mid\beta<\kappa\rangle$
for which $\{ f_\beta\mid \beta<\kappa\}$ is not bounded with respect to $<_{cl}$.
Define $c:\lambda\times\kappa\rightarrow\lambda$ via $c(\eta,\beta):=f_\beta(\eta)$.
Then, for every function $g:\lambda\rightarrow\lambda$,
for a tail of $\beta<\kappa$, $S_\beta(g):=\{\eta<\lambda\mid g(\eta)\le c_\beta(\eta)\}$ is stationary.
We claim that $c$ is strongly unbounded.
Towards a contradiction, suppose that this is not the case, as witnessed by a cofinal set $B$.
Then, we may define a function $F:{}^{<\lambda}\lambda\rightarrow\lambda$ by letting for all $\eta<\lambda$ and $t:\eta\rightarrow\lambda$,
$$F(t):=\sup\{ c_\beta(\eta)\mid \beta\in B, t\s c_\beta\}+1.$$
Now, pick a corresponding oracle $g:\lambda\rightarrow\lambda$ such that for every function $f:\lambda\rightarrow\lambda$,
the following set covers a club:
$$C_f:=\{\eta<\lambda\mid F(f\restriction\eta)\le g(\eta)\}.$$
Pick $\beta\in B$ such that $S_\beta(g)$ is stationary.
Then, find $\eta\in S_\beta(g)\cap C_{c_\beta}$. 
Altogether, $c_\beta(\eta)<F(c_\beta\restriction\eta)\le g(\eta)\le c_\beta(\eta)$.
This is a contradiction.
\end{proof}

\begin{lemma}\label{lemma27} Suppose that $\lambda$ is a regular uncountable cardinal and $\dl_\lambda$ holds.

If $\kappa=\mathfrak b_\lambda=\mathfrak d_\lambda$,
then there exists a strongly unbounded coloring $c:\lambda\times\kappa\rightarrow\lambda$.
\end{lemma}
\begin{proof} As $\kappa=\mathfrak b_\lambda=\mathfrak d_\lambda$,
it is possible to construct a coloring $c:\lambda\times\kappa\rightarrow\lambda$ with the property that for every function $g:\lambda\rightarrow\lambda$,
for a tail of $\beta<\kappa$, $$C_\beta(g):=\{\eta<\lambda\mid g(\eta)\le c_\beta(\eta)\}$$ is co-bounded in $\kappa$.\footnote{See \cite[\S6]{paper53}: for $\lambda$ regular, $\mathfrak b_\lambda=\mathfrak d_\lambda=\kappa$ implies that 
$\ubd([\lambda]^\lambda,J^{\bd}[\kappa],\lambda)$ holds.}
We claim that $c$ is strongly unbounded.
Towards a contradiction, suppose that this is not the case, as witnessed by a cofinal set $B$.
Then, we may define a function $F:{}^{<\lambda}\lambda\rightarrow\lambda$ by letting for all $\eta<\lambda$ and $t:\eta\rightarrow\lambda$,
$$F(t):=\sup\{ c_\beta(\eta)\mid \beta\in B, t\s c_\beta\}+1.$$

As $\dl_\lambda$ holds, we may pick a corresponding oracle $g:\lambda\rightarrow\lambda$ such that for every function $f:\lambda\rightarrow\lambda$,
the following set is stationary:
$$S_f:=\{\eta<\lambda\mid F(f\restriction\eta)\le g(\eta)\}.$$
Pick $\beta\in B$ such that $C_\beta(g)$ covers a club.
Then, find $\eta\in C_\beta(g)\cap S_{c_\beta}$. 
Altogether, $c_\beta(\eta)<F(c_\beta\restriction\eta)\le g(\eta)\le c_\beta(\eta)$.
This is a contradiction.
\end{proof}

\subsection{Tightness}
\begin{lemma} Suppose that $\mathfrak d=\mathfrak c=\kappa$.
Then there exists a tight strongly unbounded coloring $c:\omega\times\kappa\rightarrow\omega$.
\end{lemma}
\begin{proof} It is easy to construct an unbounded coloring $c:\omega\times \mathfrak d\rightarrow\omega$ (see \cite[\S6]{paper53}).
By Lemma~\ref{lemma23}, $c$ is moreover strongly unbounded.
As $\mathcal T_c\s\mathcal P({}^{<\omega}\omega)$, it follows that $|T_c|\le\mathfrak c$.
So, if $\mathfrak d=\mathfrak c=\kappa$, then $c$ is tight.
\end{proof}

\begin{cor} Suppose that $\lambda=\lambda^{<\lambda}$ is an infinite cardinal satisfying any of the following:
\begin{itemize}
\item $\lambda=\aleph_0$, or
\item $\lambda=\aleph_1$ and $\diamondsuit_\lambda$ holds, or
\item $\lambda>\aleph_1$ is a successor cardinal, or
\item $\lambda\ge\beth_\omega$, or
\item $\lambda$ is strongly inaccessible.
\end{itemize}

If $\kappa=\mathfrak b_\lambda=2^\lambda$,
then there exists a tight strongly unbounded coloring $c:\lambda\times\kappa\rightarrow\lambda$.
\end{cor}
\begin{proof} As in the proof of the previous lemma, the fact that $\kappa=2^{\lambda^{<\lambda}}$ 
implies that any strongly unbounded coloring $c:\lambda\times\kappa\rightarrow\lambda$ is tight.
Assuming $\kappa=\mathfrak b_\lambda$, it is also easy to obtain an unbounded coloring $c:\lambda\times\kappa\rightarrow\lambda$.
So the heart of the matter is to get a strongly unbounded one.
Lemma~\ref{lemma23} takes care of the first and last bullet.
The remaining bullets follow from Lemma~\ref{lemma27} together with Fact~\ref{dlfacts}.
\end{proof}

\begin{defn} A \emph{$\kappa$-Luzin subset of ${}^\lambda\lambda$}
is a subset $L\s{}^\lambda\lambda$ of size $\kappa$ having the property that for every $B\in[L]^\kappa$,
there exists $t\in{}^{<\lambda}\lambda$ such that, for every $t'\in{}^{<\lambda}\lambda$ extending $t$,
there exists an element of $B$ extending $t'$.
\end{defn}

It is well-known that $\ma$ implies the existence of a  $\mathfrak c$-Luzin subset of ${}^\omega\omega$. More generally, $\cov(\mathcal M)=\cof(\mathcal M)=\kappa$ entails the existence of a $\kappa$-Luzin subset of ${}^\omega\omega$.
Also, the following fact is standard:

\begin{fact}[Luzin] For every infinite cardinal $\lambda=\lambda^{<\lambda}$,
if $2^\lambda=\lambda^+$, then there exists a $\lambda^+$-Luzin subset of ${}^\lambda\lambda$.
\end{fact}

\begin{lemma} Suppose that there exists a $\kappa$-Luzin subset of ${}^{\lambda}\lambda$, with $\kappa$ regular.

If $\lambda^{<\lambda}<\kappa$, then there exists a tight strongly unbounded coloring $c:\lambda\times\kappa\rightarrow\lambda$.
\end{lemma}
\begin{proof} Fix an injective enumeration $\vec g=\langle g_\beta\mid \beta<\kappa\rangle$ of a $\kappa$-Luzin subset of ${}^\lambda\lambda$.
Let $c:\lambda\times\kappa\rightarrow\lambda$ denote the unique coloring such that $c_\beta=g_\beta$ for all $\beta<\kappa$.
\begin{claim} $c$ is strongly unbounded.
\end{claim} 
\begin{proof} Let $B\in [\kappa]^\kappa$; we need to find 
$\eta<\lambda$ and a map $t:\eta\rightarrow\lambda$ such that 
$\sup\{ c_\beta(\eta)\mid \beta\in B\  \&\ t\s c_\beta\}=\lambda$.
As $\im(\vec g)$ is a $\kappa$-Luzin subset of ${}^\lambda\lambda$, 
fix some $t\in {}^{<\lambda}\lambda$ such that, for every $t'\in{}^{<\lambda}\lambda$ extending $t$,
there exists $l\in \{g_\beta \mid \beta \in B\}$ extending $t'$. Set $\eta:=\dom(t)$.
Then for every $\gamma<\lambda$, we can find $\beta \in B$ such that $g_\beta$ extends $t{}^\smallfrown\langle\gamma\rangle$.
Altogether, $\{ c_\beta(\eta)\mid \beta\in B\  \&\ t\s c_\beta\}=\lambda$.	
\end{proof}
For every $s\in{}^{<\lambda}\lambda$, denote $\mathbb T_s:=\{ t\in{}^{<\lambda}\lambda\mid s\s t\text{ or }t\s s\}$.
\begin{claim} For every $\mathbb T\in\mathcal T_c$,
there exists  $s\in{}^{<\lambda}\lambda$ such that $\mathbb T\supseteq\mathbb T_s\in\mathcal T_c$.
\end{claim} 
\begin{proof} Let $\mathbb T\in\mathcal T_c$,
so that $[\mathbb T]_c=\{ \beta<\kappa\mid \forall\eta<\lambda\,(c_\beta\restriction\eta\in\mathbb  T)\}$
is cofinal in $\kappa$. For every $t\in{}^{<\lambda}\lambda$, write $A_t:=\{ \beta\in [\mathbb T]_c\mid t\s c_\beta\}$.
As  $\lambda^{<\lambda}<\cf(\kappa)=\kappa$, the set $N:=\bigcup\{ A_t\mid t\in{}^{<\lambda}\lambda, |A_t|<\kappa\}$ has size $<\kappa$.
In particular, $B:= [\mathbb T]_c\setminus N$ has size $\kappa$.
As $\im(\vec g)$ is a $\kappa$-Luzin subset of ${}^\lambda\lambda$, 
fix some $s\in {}^{<\lambda}\lambda$ such that, for every $s'\in{}^{<\lambda}\lambda$ extending $s$,
there exists $\beta \in B$ such that $s'\s g_\beta$. As $B\s[\mathbb T]_c$, it follows that $\mathbb T_s\s\mathbb T$.
	
	Finally, to show that $[\mathbb T_s]_c=\{ \beta<\kappa\mid \forall\eta<\lambda\,(c_\beta\restriction\eta\in\mathbb  T_s)\}$  is in $\mathcal T_c$,
	we need to prove that $\sup([\mathbb T_s]_c)=\kappa$.
Recalling that $s$ extends $g_\beta$ for some $\beta\in B\s(\kappa\setminus N)$, we infer that $|A_s|=\kappa$.
As $[\mathbb T_s]_c$ clearly covers $A_s$, we infer that $\sup([\mathbb T_s]_c)=\kappa$.
\end{proof}
In particular, $\cf(\mathcal T_c,{\supseteq})\le\lambda^{<\lambda}$. So, we are done.
\end{proof}

\section{Theorem~A}\label{Sec3}
\begin{defn}[{\cite[\S3.3]{paper36}}] Let $\lambda<\kappa$ be a pair of infinite cardinals, $e:[\kappa]^2\rightarrow\lambda$ be a coloring, and $S$ be a subset of $\kappa$.
\begin{enumerate}
\item $e$ is \emph{$S$-coherent} iff for all $\beta\le\gamma<\delta<\kappa$ with $\beta\in S$,
$$\sup\{\xi<\beta\mid e(\xi,\gamma)\neq e(\xi,\delta)\}<\beta;$$
\item $\partial(e):=\{ \alpha\in\acc(\kappa)\mid \forall \gamma<\kappa\forall \nu<\lambda\,\sup\{ \xi<\alpha\mid e(\xi,\gamma)\le\nu\}<\alpha\}$.
\end{enumerate}
\end{defn}
\begin{fact}[{\cite[Lemma~3.31]{paper36}}]\label{nonreflecting} Let $\lambda<\kappa$ be a pair of infinite regular cardinals. 

For a stationary subset $S\s E^\kappa_{\lambda}$, the following are equivalent:
\begin{itemize}
\item $S$ is nonreflecting;
\item There exists an $S$-coherent coloring $e:[\kappa]^2\rightarrow\lambda$  such that $\partial(e)\supseteq S$.
\end{itemize}
\end{fact}

\begin{theorem}\label{maintheorem} Suppose:
\begin{enumerate}
\item $\theta\le\lambda<\kappa$ are infinite cardinals, with $\lambda,\kappa$ regular,
\item $c:\lambda\times\kappa\rightarrow\lambda$ is a strongly unbounded coloring,
\item $c$ is tight. Furthermore, $(\cf(\mathcal T_c,{\supseteq}))^{<\theta}\le\kappa$, and
\item $\mathcal S$ is a partition of some nonreflecting stationary subset of $E^{\kappa}_{\lambda}$ into stationary sets.
\end{enumerate}

Then there exists a club $C\s\kappa$ such that $\clubsuit_{\ad}(\{ S\cap C\mid S\in \mathcal S\},\allowbreak\lambda,{<}\theta)$ holds.
If either $\kappa=\lambda^+$ or $\lambda^{<\lambda}=\lambda$, then $C$ can moreover be taken to be whole of $\kappa$.
\end{theorem} 
\begin{proof}  The proof is an elaboration of a construction from \cite[\S2]{TodPartTop}.
Let $\mathcal T$ be a dense subfamily of $\mathcal T_c$ of minimal size.
Let $\seq_{<\theta}(\mathcal T)$ denote the collection of all nonempty sequences of elements of $\mathcal T$ of length $<\theta$.
By Clause~(3) above, $|\seq_{<\theta}(\mathcal T)|\le\kappa$.
Let $\mathcal S$ be a given partition of some nonreflecting stationary subset $\mathbb{S}$ of $E^{\kappa}_{\lambda}$ into stationary sets.
Then, for every $S\in\mathcal S$, let $\langle S_\sigma\mid \sigma\in\seq_{<\theta}(\mathcal T)\rangle$ be a partition of $S$ into stationary sets.

\begin{claim}\label{claim3331} The set $\Sigma:=\{c_\beta\restriction\Lambda\mid \beta<\kappa, \Lambda<\lambda\}$ has size $<\kappa$.
\end{claim}
\begin{proof} Suppose not. Since $\kappa$ is a regular cardinal greater than $\lambda$, it follows that there exist $B\in[\kappa]^\kappa$ and $\Lambda<\lambda$ on which the map $\beta\mapsto c_\beta\restriction\Lambda$
is injective. By possibly shrinking $B$ further, we may also assume the existence of some $\epsilon<\lambda$ such that $c_\beta[\Lambda]\s\epsilon$ for all $\beta\in B$.
But then $c$ cannot be strongly unbounded. Indeed, for every $t:\eta\rightarrow\lambda$, if $\eta<\Lambda$, then $\{ c_\beta(\eta)\mid \beta\in B, t\s c_\beta\}\s \epsilon$,
and if $\eta\in[\Lambda,\lambda)$, then $|\{ c_\beta(\eta)\mid \beta\in B, t\s c_\beta\}|\le1$.
\end{proof}
Set $T^*:=\{ t\in \Sigma\mid |\{\beta<\kappa\mid t\s c_\beta\}|=\kappa\}$.
By Claim~\ref{claim3331}, $|T^*|<\kappa$, so we may fix a surjection $f:\kappa\rightarrow T^*$ with the property that for every $\epsilon<\kappa$,
$\{f(\xi+1)\mid \epsilon<\xi<\epsilon+|T^*|\}=T^*$.
By Claim~\ref{claim3331}, we may also fix a large enough ordinal $\varrho<\kappa$ such that $\bigcup\{\{ \beta<\kappa \mid t\s c_\beta\}\mid  t\in \Sigma\setminus T^*\}\s\varrho$.

Next, we turn to recursively define an injective matrix $\langle \beta_{\alpha, j}\mid \alpha<\kappa,  j<\zeta(\alpha)\rangle$ of ordinals in $\kappa$, as follows.
Suppose that $\alpha<\kappa$ and that 
$\langle \beta_{\bar\alpha,\bar j}\mid\bar\alpha<\alpha,\bar j<\zeta(\bar\alpha)\rangle$ has already been defined.
\begin{itemize}
\item[$\br$] If $\{ (S,\sigma)\in\mathcal S\times\seq_{<\theta}(\mathcal T)\mid \alpha\in S_\sigma\}$ is nonempty,
then it is a singleton, so let $(S,\sigma)$ denote its unique element.
Write $\sigma=\langle \mathbb T_ j\mid  j<\zeta\rangle$. Set $\zeta(\alpha):=\zeta$,
and then, by recursion on $ j<\zeta(\alpha)$, set 
\begin{equation}\label{defofbetaalpha}
\tag{$\star$}\beta_{\alpha,j}:=\min([\mathbb T_ j]_c\setminus(\varrho\cup\{ \beta_{\bar\alpha,\bar j},\beta_{\alpha,j'}\mid\bar\alpha<\alpha , \bar  j\le \zeta(\bar\alpha),  j'< j\})).
\end{equation}
\item[$\br$] Otherwise, set $\zeta(\alpha):=1$ and 
\begin{equation}\label{defofbetaalpha2}
\tag{$\star\star$}\beta_{\alpha,0}:=\min(\{\beta<\kappa\mid f(\alpha)\s c_\beta\}\setminus(\varrho\cup\{\beta_{\bar\alpha,\bar j}\mid \bar\alpha<\alpha, \bar j<\zeta(\bar\alpha)\})).
\end{equation}
\end{itemize}

Having constructed the above injective matrix of ordinals in $\kappa$,
we derive a corresponding injective matrix 
$\vec d=\langle d_{\alpha, j}\mid\alpha<\kappa, j<\zeta(\alpha)\rangle$  by setting $d_{\alpha, j}:=c_{\beta_{\alpha, j}}$.

For all $x\neq y$ in ${}^\lambda\lambda$, denote $\Delta(x,y):=\min\{\eta<\lambda\mid x(\eta)\neq y(\eta)\}$.
As $\lambda$ is regular, for every $x\in {}^\lambda\lambda$, we may attach a strictly increasing function $\hat x:\lambda\rightarrow\lambda$ 
satisfying $\hat x(\eta)\ge x(\eta)$ for all $\eta<\lambda$.

Next, as $\mathbb{S}$ is a nonreflecting stationary subset of $E^{\kappa}_{\lambda}$, 
by Fact~\ref{nonreflecting}, we may fix a coloring $e:[\kappa]^2\rightarrow\lambda$ that is $\mathbb{S}$-coherent
and such that $\mathbb{S}\s \partial(e)$.

Fix a surjection $\varsigma:\kappa\rightarrow\lambda$ such that $\partial(e)\cap \varsigma^{-1}\{i\}$ is stationary for every $i<\lambda$.
Next, for every $\alpha<\kappa$ and $ j<\zeta(\alpha)$, define a map $h_{\alpha, j}:\alpha\rightarrow\lambda$ via:
$$h_{\alpha, j}(\xi):=\varsigma(\min\{ \gamma\in(\xi,\alpha]\mid \gamma=\alpha\text{ or }e(\gamma,\alpha)\le{{\Delta(d_{\xi,0},d_{\alpha, j})}}\}).$$

For all $\alpha<\kappa$ and $i<\lambda$,
let $A^i_{\alpha,*}:=\bigcup_{ j<\zeta(\alpha)}A^i_{\alpha, j}$,
where for every $ j<\zeta(\alpha)$:
$$A^i_{\alpha, j}:=\{\xi<\alpha\mid  h_{\alpha, j}(\xi)=i\ \&\ e(\xi,\alpha)\le \hat{d}_{\alpha, j}(\Delta(d_{\xi,0},d_{\alpha, j}))\}.$$

\begin{claim}\label{claim331} Suppose that $\alpha\in\partial(e)$, $\beta<\kappa$, $j<\zeta(\alpha), j'<\zeta(\beta)$, and $i,i'<\lambda$.

If $(\alpha,j)\neq(\beta,j')$. Then $\sup(A_{\alpha, j}^i\cap A^{i'}_{\beta, j'})<\alpha$.
\end{claim}
\begin{proof} Suppose that $(\alpha,j)\neq(\beta,j')$,
and then let $\eta:=\Delta(d_{\alpha, j},d_{\beta, j'})$.
Towards a contradiction, suppose that $A^i_{\alpha, j}\cap A^{i'}_{\beta, j'}$ is cofinal in $\alpha$.
Set  $\nu:=\hat d_{\alpha, j}(\eta)$.
As $\alpha\in\partial(e)$, the following set is cofinal in $\alpha$:
$$Y:=\{ \xi\in A^i_{\alpha, j}\cap A^{i'}_{\beta, j'}\mid e(\xi,\alpha)>\nu\}.$$

For every $\xi\in Y$, $\hat d_{\alpha, j}(\eta)=\nu<e(\xi,\alpha)\le \hat d_{\alpha, j}(\Delta( d_{\xi,0},d_{\alpha, j}))$,
so since $\hat d_{\alpha, j}$ is strictly increasing, $\Delta( d_{\xi,0},d_{\alpha, j})>\eta=\Delta(d_{\alpha, j},d_{\beta, j'})$,
and hence $\Delta(d_{\xi,0},d_{\beta, j'})=\eta$.
Set $\tau:=\hat d_{\beta, j'}(\eta)$. As $Y\s \alpha\cap A^{i'}_{\beta, j'}$, altogether $Y\s \{\xi<\alpha\mid e(\xi,\beta)\le \tau\}$.
As $\alpha\in\partial(e)$, $Y$ is bounded in $\alpha$. This is a contradiction.
\end{proof}
 
\begin{claim}\label{claim332} Let $\alpha\in\mathbb S$ and $i\neq i'$ in $\lambda$. Then $\sup(A^{i}_{\alpha,*}\cap A^{i'}_{\alpha,*})<\alpha$.
\end{claim}
\begin{proof}  Suppose not. As
$\zeta(\alpha)<\theta\le \lambda=\cf(\alpha)$, there must exist $ j,j'<\zeta(\alpha)$ such that 
$\sup(A^i_{\alpha, j}\cap A^{i'}_{\alpha, j'})=\alpha$. As $\alpha\in\mathbb S\s\partial(e)$, Claim~\ref{claim331} implies that $j=j'$.
But it is evident that $A^i_{\alpha, j}$ and $A^{i'}_{\alpha, j}$ are disjoint.
\end{proof}

For all $\alpha\in\mathbb S$ and $i<\lambda$, let $$A^i_\alpha:=A^i_{\alpha,*}\setminus\bigcup_{i'<i} A^{i'}_{\alpha,*}.$$
Clearly, $\langle A^i_\alpha\mid i<\lambda\rangle$ consists of pairwise disjoint subsets of $\alpha$.
\begin{claim} Let $(\alpha,\beta)\in[\mathbb S]^2$ and $i,i'<\lambda$. Then $\sup(A^i_\alpha\cap A^{i'}_\beta)<\alpha$.
\end{claim}
\begin{proof} Suppose not. In particular, $\sup(A^i_{\alpha, *}\cap A^{i'}_{\beta, *})=\alpha$.
However, $\zeta(\alpha),\zeta(\beta)<\theta\le \lambda=\cf(\alpha)$, so there must exist $ j<\zeta(\alpha)$ and $ j'<\zeta(\beta)$ such that 
$\sup(A^i_{\alpha, j}\cap A^{i'}_{\beta, j'})=\alpha$, contradicting Claim~\ref{claim331}.
\end{proof}

Now, we turn to inspect the guessing features of the matrix $\langle A_\alpha^i\mid \alpha\in\mathbb S, i<\lambda\rangle$.

\begin{claim}\label{claim312} Let $S\in\mathcal S$, and let $\langle X_ j\mid  j<\zeta\rangle$ be any sequence of cofinal subsets of $\kappa$ with $0<\zeta<\theta$.
Then $\{ \alpha\in S\mid \forall i<\lambda\forall j<\zeta\,\sup(A^i_\alpha\cap X_ j)=\alpha\}$ is stationary.
\end{claim}
\begin{proof} For all $ j<\zeta$ and $t\in \Sigma$, denote $X_ j^t:=\{ \xi\in X_ j\mid t\s d_{\xi,0}\}$.
Set $$T_ j:=\{ t\in \Sigma\mid |X_ j^t|=\kappa\}.$$
By Claim~\ref{claim3331}, $N_ j:=\bigcup\{ X_ j^t\mid t\in \Sigma\setminus T_ j\}$ is the small union of sets of size $<\kappa$, so that $|X_ j\setminus N_ j|=\kappa$.
For all $\xi\in X_ j\setminus N_ j$ and $\eta<\lambda$,
$c_{\beta_{\xi,0}}\restriction\eta=d_{\xi,0}\restriction\eta\in T_ j$,
so that $[T_ j]_c$ covers $\{ \beta_{\xi,0}\mid \xi\in X_ j\setminus N_ j\}$, and
hence $T_ j\in\mathcal T_c$. 
Recalling that $\mathcal T$ is dense in $\mathcal T_c$,
we may now pick $\mathbb T_ j\in\mathcal T$ with $\mathbb T_j\s T_j$.
In particular, $\sigma:=\langle \mathbb T_ j\mid j<\zeta\rangle$ is in $\seq_{<\theta}(\mathcal T)$,
and $S_\sigma$ is stationary.

Towards a contradiction, suppose that $\{ \alpha\in S\mid \forall i<\lambda\forall j<\zeta\,\sup(A^i_\alpha\cap X_ j)=\alpha\}$ is nonstationary.
As $S_\sigma$ is a stationary subset of $S$, we may fix $i<\lambda$ and $j<\zeta$
for which the following set is stationary:
$$S^0:=\{\alpha\in S_\sigma\mid \sup(A^i_\alpha\cap X_ j)<\alpha\}.$$
For every $\alpha\in S^0$, since $i<\lambda=\cf(\alpha)$ and since $A^i_\alpha=A^i_{\alpha,*}\setminus\bigcup_{i'<i} A^{i'}_{\alpha,*}$,
Claim~\ref{claim332} implies that $\sup(A^i_{\alpha,*}\cap X_ j)<\alpha$.
In particular, $\sup(A^i_{\alpha,j}\cap X_ j)<\alpha$.
So by Fodor's lemma, we may fix an $\epsilon<\kappa$ such that the following set is stationary:
$$S^1:=\{\alpha\in S_\sigma\mid \sup(A^i_{\alpha,j}\cap X_ j)=\epsilon<\alpha\}.$$

By the choice of the map $\varsigma$, the set $\Gamma$ of all $\gamma\in\partial(e)\cap\varsigma^{-1}\{i\}$
for which there exists an elementary submodel $M_\gamma\prec H_{\kappa^+}$
containing $\{\vec d,X_ j,\Sigma\}$ and satisfying $\gamma=M_\gamma\cap\kappa$ is stationary.
Fix $\delta\in\mathbb S\cap\acc^+(\Gamma\setminus\epsilon)$.
As $\delta\in\mathbb S$ and $e$ is $\mathbb{S}$-coherent, we may fix $S^2\in[S^1\setminus\delta]^{\kappa}$ along with some $\varepsilon<\delta$ such that, for every $\alpha\in S^2$,
$$\{ \xi<\delta\mid e(\xi,\alpha)\neq e(\xi,\delta)\}\s\varepsilon.$$
Pick $\gamma\in\Gamma\cap\delta$ above $\max\{\epsilon,\varepsilon\}$,
and then fix a model $M_\gamma$ witnessing that $\gamma\in\Gamma$.

Put $\nu:=e(\gamma,\delta)$. 
By Claim~\ref{claim3331}, we may find a cofinal subset of $S^3\s S^2$ on which the map $\alpha\mapsto d_{\alpha, j}\restriction\nu$ is constant.

Next, as $c$ is strongly unbounded and $\{ \beta_{\alpha, j}\mid \alpha\in S^3\}$ is cofinal in $\kappa$,
we may find an ordinal $\eta<\lambda$ and a map $t:\eta\rightarrow\lambda$ such that
$$\sup\{ c_{\beta_{\alpha,j}}(\eta)\mid \alpha\in S^3, t\s c_{\beta_{\alpha,j}}\}=\lambda.$$
Equivalently, for every $\tau<\lambda$, for some $\alpha\in S^3$, $d_{\alpha, j}\restriction\eta=t$ and $d_{\alpha, j}(\eta)>\tau$.
Clearly, $\eta\ge\nu$.

Pick for a moment $\alpha^*\in S^3$ such that $t\s d_{\alpha^*, j}$.
Since $\alpha^*\in S^3\s S_\sigma$, 
Equation~\eqref{defofbetaalpha} and the definition of $\sigma$ implies that $\beta_{\alpha^*, j}$ is in $[\mathbb T_ j]_c$.
Recalling Definition~\ref{deftrees}, from $c_{\beta_{\alpha^*, j}}\restriction\eta=d_{\alpha^*, j}\restriction\eta=t$, we infer that $t\in \mathbb T_ j$. 
As $\mathbb T_ j\s T_ j$, this means that $|X_ j^t|=\kappa$.
It thus follows from $\{\vec d,X_ j,\Sigma\}\in M_\gamma$ that $\sup(X_ j^t\cap\gamma)=\gamma$.
Now, as $\gamma\in\partial(e)$, $G:=\{\bar\gamma<\gamma\mid e(\bar\gamma,\delta)\le{{\eta}}\}$ is bounded below $\gamma$.
Altogether, we may find $\xi\in X_ j^t\cap \gamma$ above $\max\{\epsilon,\varepsilon,\sup(G)\}$. 

Set $\tau:=\max\{e(\xi,\delta),d_{\xi,0}(\eta)\}$,
and then pick $\alpha\in S^3$ such that $d_{\alpha, j}\restriction\eta=t$ and $d_{\alpha, j}(\eta)>\tau$.
As $d_{\alpha, j}\restriction\eta=t=d_{\xi,0}\restriction\eta$ and $d_{\alpha, j}(\eta)>d_{\xi,0}(\eta)$, we infer that $\Delta(d_{\xi,0},d_{\alpha, j})=\eta$. 
As $\varepsilon<\xi<\delta$, altogether, $$\hat d_{\alpha, j}(\Delta(d_{\xi,0},d_{\alpha, j}))=\hat d_{\alpha, j}(\eta)\ge d_{\alpha, j}(\eta)>\tau\ge e(\xi,\delta)=e(\xi,\alpha).$$ 
Next, from $\Delta(d_{\xi,0},d_{\alpha, j})=\eta$ and the fact that $\xi>\varepsilon$,
we also infer that $$h_{\alpha, j}(\xi)=\varsigma(\min\{ \bar\gamma\in(\xi,\alpha]\mid \bar\gamma=\alpha\text{ or }e(\bar\gamma,\delta)\le{{\eta}}\}).$$
Since $e(\gamma,\delta)=\nu\le{{\eta}}$ and $\varsigma(\gamma)=i$,
it follows that if $h_{\alpha, j}(\xi)\neq i$, then there exists $\bar\gamma\in(\xi,\gamma)$ such that $e(\bar\gamma,\delta)\le{{\eta}}$,
contradicting the fact that $\xi>\sup(G)$. 
So, it is the case that $h_{\alpha, j}(\xi)=i$. Consequently, $\xi\in A^i_{\alpha,j}$.

Altogether, we established that $\xi$ is an element of $A^i_{\alpha,j}\cap X_j$ above $\epsilon$, contradicting the fact that $\alpha\in S^3\s S^1$.
\end{proof}

The next claim implies that there exists a club $C\s\kappa$ such that,
for every $\alpha\in C$, 
for every $i<\lambda$, $\sup(A^i_\alpha)=\alpha$.

\begin{claim} Let $S\s\mathbb S$ be stationary. 
Then $\{ \alpha\in S\mid \forall i<\lambda\,\sup(A^i_\alpha\cap S)=\alpha\}$ is stationary.
\end{claim}
\begin{proof} Suppose not, and fix $i<\lambda$ 
for which the following set is stationary:
$$S^0:=\{\alpha\in S\mid \sup(A^i_\alpha\cap S)<\alpha\}.$$
It follows that there exists an $\epsilon<\kappa$ such that
$$S^1:=\{\alpha\in S\mid \sup(A^i_{\alpha,0}\cap S)=\epsilon<\alpha\}.$$
Similarly to the proof of the previous claim,
find ordinals $\varepsilon<\gamma<\delta$ and a set $S^2\in[S^1\setminus\delta]^\kappa$ such that:
\begin{itemize} 
\item $\gamma,\delta\in\partial(e)$;
\item $\delta=M_\delta\cap\kappa$ for some elementary submodel $M_\delta\prec H_{\kappa^+}$ containing $\{\vec d,S,\Sigma\}$;
\item $\gamma=M_\gamma\cap\kappa$ for some elementary submodel $M_\gamma\prec H_{\kappa^+}$ containing $\{\vec d,S,\Sigma\}$;
\item for every $\alpha\in S^2$, $\{ \xi<\delta\mid e(\xi,\alpha)\neq e(\xi,\delta)\}\s\varepsilon$.
\end{itemize}

Put $\nu:=e(\gamma,\delta)$. 
Then find a cofinal subset of $S^3\s S^2$ on which the map $\alpha\mapsto d_{\alpha,0}\restriction\nu$ is constant.
As $\{ \beta_{\alpha,0}\mid \alpha\in S^3\}$ is cofinal in $\kappa$,
the choice of the coloring $c$ provides an ordinal $\eta<\lambda$ and a map $t:\eta\rightarrow\lambda$ such that, 
for every $\tau<\lambda$, for some $\alpha\in S^3$, $d_{\alpha,0}\restriction\eta=t$ and $d_{\alpha,0}(\eta)>\tau$.
The same analysis is true for any final segment of $S^3$ and hence, by Clause~(1) and the pigeonhole principle, 
we may fix some $t:\eta\rightarrow\lambda$ such that, 
for every $\tau<\lambda$, for cofinally many $\alpha\in S^3$, $d_{\alpha,0}\restriction\eta=t$ and $d_{\alpha,0}(\eta)>\tau$.
Clearly, $\eta\ge\nu$.

As $\{\vec d,S,\Sigma\}\in M_\delta$ and $S^3\cap M_\delta=\emptyset$, by elementarity, the set of $\xi\in S\cap M_\delta$ such that $d_{\xi,0}\restriction\eta=t$
is cofinal in $\gamma$.
As $\gamma\in\partial(e)$, $G:=\{\bar\gamma<\gamma\mid e(\bar\gamma,\delta)\le{{\eta}}\}$ is bounded below $\gamma$.
So, we may find $\xi\in S\cap \gamma$ above $\max\{\epsilon,\varepsilon,\sup(G)\}$ such that $d_{\xi,0}\restriction\eta=t$.
Set $\tau:=\max\{e(\xi,\delta),d_{\xi,0}(\eta)\}$,
and then pick $\alpha\in S^3$ such that $d_{\alpha,0}\restriction\eta=t$ and $d_{\alpha,0}(\eta)>\tau$.
From this point on, a verification identical to that of Claim~\ref{claim312}
shows that $\xi$ is an element of $A^i_{\alpha,0}\cap S$ above $\epsilon$, contradicting the fact that $\alpha\in S^3\s S^1$.
\end{proof}

In summary, we have shown that there exists a club $C\s\kappa$ such that:
\begin{enumerate}
\item[(a)] For every $\alpha\in C$, for every $i<\lambda$, $\sup(A^i_\alpha)=\alpha$;
\item[(b)] For every $S\in\mathcal S$, for every sequence $\langle X_ j\mid  j<\zeta\rangle$ of cofinal subsets of $\kappa$ with $0<\zeta<\theta$,
the set $\{ \alpha\in S\mid \forall i<\lambda\forall j<\zeta\,\sup(A^i_\alpha\cap X_ j)=\alpha\}$ is stationary;
\item[(c)] For all $(\alpha,\beta)\in[\mathbb S]^2$ and $i,i'<\lambda$, 
$\sup(A^i_\alpha\cap A^{i'}_\beta)<\alpha$.
\end{enumerate}

Suppose now that either $\kappa=\lambda^+$ or $\lambda^{<\lambda}=\lambda$, 
and let us prove that $\clubsuit_{\ad}(\mathcal S,\lambda,{<}\theta)$ holds. For this, it suffices to define for every $\alpha\in \mathbb S\setminus C$,
a sequence $\langle a_\alpha^i\mid i<\lambda\rangle$ of pairwise disjoint cofinal subsets of $\alpha$ such that
the amalgam of $\langle \langle a_\alpha^i\mid i<\lambda\rangle\mid \alpha\in \mathbb S\setminus C\rangle $
and $\langle \langle A_\alpha^i\mid i<\lambda\rangle\mid \alpha\in \mathbb S\cap C\rangle $ will form an almost-disjoint system.
To this end, let $\alpha\in \mathbb S\setminus C$.
\begin{claim} For every $\eta<\lambda$, $\sup\{\xi<\alpha\mid \Delta(d_{\xi,0},d_{\alpha,0})\ge\eta\}=\alpha$.
\end{claim}
\begin{proof} Let $\eta<\lambda$, $t:=d_{\alpha,0}\restriction\eta$ and $\epsilon<\alpha$;
we need to find $\xi$ with $\epsilon<\xi<\alpha$ such that $t\s d_{\xi,0}$.
Now, recall that by the construction of $\vec d$, $d_{\alpha,0}=c_\beta$ for some ordinal $\beta\in\kappa\setminus\varrho$.
Consequently, $t=c_\beta\restriction\eta\in T^*$.
So since either $\kappa=\lambda^+$ or $\lambda^{<\lambda}=\lambda$,
Claim~\ref{claim3331} implies that $|T^*|\le\lambda=\cf(\alpha)$. Then, since the surjection $f$ was chosen to satisfy 
$\{f(\xi+1)\mid \epsilon<\xi<\epsilon+|T^*|\}=T^*$,
we may find some $\xi$ with $\epsilon<\xi<\xi+1<\alpha$ such that $f(\xi+1)=t$.
As $\bigcup\mathcal S\s \mathbb S$,
we get from Equation~\eqref{defofbetaalpha2} that $d_{\xi+1,0}=c_{\beta_{\xi+1,0}}\supseteq t$,
and hence $\Delta(d_{\xi+1,0},d_{\alpha,0})\ge\eta$. 
\end{proof}

Using the preceding claim, fix a strictly increasing sequence $\langle \xi_\eta^\alpha\mid \eta<\lambda\rangle$ of ordinals, converging to $\alpha$,
such that, for every $\eta<\lambda$, $\Delta(d_{\xi_\alpha^\alpha,0},d_{\alpha,0})\ge\eta$.
Then, let $\langle a_\alpha^i\mid i<\lambda\rangle$ be some partition of $\{\xi_\eta^\alpha\mid \eta<\lambda\}$ into $\lambda$ many sets of size $\lambda$.

As each $a_\alpha^i$ has order-type $\lambda$, the verification of almost-disjointness of the merged systems boils down to verifying the following case.
\begin{claim} Let $\alpha\in \mathbb S\setminus C$ and $\beta\in \mathbb S\cap C$ above $\alpha$.
Let $i,{i'}<\lambda$. Then $\sup(a_\alpha^i\cap A_\beta^{i'})<\alpha$.
\end{claim}
\begin{proof} Suppose not. Fix $ j'<\zeta(\beta)$ such that 
$a_{\alpha}^i\cap A^{i'}_{\beta, j'}$ is cofinal in $\alpha$.
Set $\eta:=\Delta(d_{\alpha,0},d_{\beta, j'})$. By the choice of $a_\alpha^i$, $\{ \xi\in a^i_\alpha\mid \Delta(d_{\xi,0},d_{\alpha,0})\le\eta\}$ is bounded in $\alpha$,
and hence the following set is cofinal in $\alpha$:
$$Y:=\{ \xi\in a_{\alpha}^i\cap A^{i'}_{\beta, j'}\mid \Delta(d_{\xi,0},d_{\alpha,0})>\eta\}.$$

For every $\xi\in Y$,  $\Delta( d_{\xi,0},d_{\alpha,0})>\eta=\Delta(d_{\alpha,0},d_{\beta, j'})$,
and hence $\Delta(d_{\xi,0},d_{\beta, j'})=\eta$.
Set $\tau:=\hat d_{\beta, j'}(\eta)$. As $Y\s\alpha\cap A_{\beta, j'}$, altogether $Y\s \{\xi<\alpha\mid e(\xi,\beta)\le \tau\}$.
As $\alpha\in \mathbb{S}\s \partial(e)$, $Y$ is bounded in $\alpha$. This is a contradiction.
\end{proof}
This completes the proof.
\end{proof}

\subsection{Variations}
A second read of the proof of Theorem~\ref{maintheorem}
makes it clear that the conclusion remains valid even after relaxing Clause~(3) in the hypothesis 
to $\cov(\cf(\mathcal T_c,{\supseteq}),\allowbreak\lambda,\theta,2)\le\kappa$.
In the other direction,
by waiving Clause~(3) completely, the above proof yields the following:
\begin{theorem}\label{thm34} Suppose $\lambda<\kappa$ is a pair of infinite regular cardinals,
and $\mathbb S$ is a nonreflecting stationary subset of $E^\kappa_\lambda$.

If there exists a strongly unbounded coloring $c:\lambda\times\kappa\rightarrow\lambda$,
then there exists a club $C\s\kappa$ and a matrix
$\langle A_\alpha^i\mid \alpha\in\mathbb S\cap C, i<\alpha\rangle$ such that:
\begin{enumerate}
\item For every $\alpha\in\mathbb S\cap C$, 
$\langle A^i_\alpha\mid i<\alpha\rangle$ is a sequence 
of pairwise disjoint cofinal subsets of $\alpha$;
\item For every stationary $S\s\mathbb S$,
there are stationarily many $\alpha\in S\cap C$ such that $\sup(A^i_\alpha\cap S)=\alpha$ for all $i<\alpha$; 
\item For all $(\alpha,\alpha')\in[\mathbb S\cap C]^2$, $i<\alpha$ and $i'<\alpha'$, 
$\sup(A^i_\alpha\cap A^{i'}_{\alpha'})<\alpha$.
\end{enumerate} 

In the special case that $\kappa=\lambda^+$ or $\lambda^{<\lambda}=\lambda$, one can take $C$ to be whole of $\kappa$.\qed
\end{theorem} 

Let $\clubsuit_{\ad^*}(\mathcal S,\mu,{<}\theta)$ 
denote the strengthening of $\clubsuit_{\ad}(\mathcal S,\mu,{<}\theta)$
obtained by replacing Clause~(3) of Definition~\ref{clubad} by:
\begin{enumerate}
\item[($3^*$)] For all $A\neq A'$ from $\bigcup_{S\in\mathcal S}\bigcup_{\alpha\in S}\mathcal A_\alpha$, $|A\cap A'|<\cf(\sup(A))$.
		\end{enumerate} 
In the special case that $\kappa=\lambda^+$, one can use in the proof of Theorem~\ref{maintheorem} a \emph{locally small} coloring $e:[\kappa]^2\rightarrow\lambda$ (such as the map $\rho_1$ from \cite[\S6.2]{TodWalks}), and then get:

\begin{theorem} Suppose:
\begin{enumerate}
\item $\lambda$ is an infinite regular cardinal,
\item there exists a tight strongly unbounded coloriung $c:\lambda\times\lambda^+\rightarrow\lambda$, and
\item $\mathcal S$ is a partition of $E^{\lambda^+}_{\lambda}$ into stationary sets.
\end{enumerate}

Then $\clubsuit_{\ad^*}(\mathcal S,\lambda,{<}\lambda)$ holds.\qed
\end{theorem}

\section{Theorem~B}\label{Sec4}
In this section, we give two sufficient conditions for a strong form of $\clubsuit_{\ad}$ to hold. 
The strong form under discussion is a double strengthening of $\clubsuit_{\ad}(\mathcal S,\mu,{<}\kappa)$,
and it reads as follows.
\begin{defn}  Let $\mathcal S$ be a collection of stationary subsets of a regular uncountable cardinal $\kappa$,
and $\mu$ be a nonzero cardinal $<\kappa$.
The principle 
$\clubsuit_{\ad^*}(\mathcal S,\mu,\kappa)$ asserts the existence of a sequence $\langle \mathcal A_\alpha\mid \alpha\in\bigcup\mathcal S\rangle$ such that:
\begin{enumerate}
\item For every $\alpha\in\acc(\kappa)\cap\bigcup\mathcal S$, $\mathcal A_\alpha$ is a pairwise disjoint family of $\mu$ many cofinal subsets of $\alpha$;
\item For every sequence $\langle B_i\mid i<\kappa\rangle$ of cofinal subsets of $\kappa$, for every $S\in\mathcal S$, there are stationarily many $\alpha\in S$ such that,
for all $A\in\mathcal A_\alpha$ and $i<\alpha$,  $\sup(A\cap B_i)=\alpha$;
\item For all $A\neq A'$ from $\bigcup_{S\in\mathcal S}\bigcup_{\alpha\in S}\mathcal A_\alpha$, $|A\cap A'|<\cf(\sup(A))$.
\end{enumerate} 
\end{defn}

An inspection of the proof \cite[\S3]{MR4068775} yields the following useful fact:

\begin{fact}[{\cite{MR4068775}}]\label{claim512} 
For an infinite cardinal $\lambda$, the following are equivalent:
\begin{enumerate}
\item $\stick(\lambda^+)$ holds, i.e.,
there exists a sequence $\langle x_\beta\mid \beta<\lambda^+\rangle$ 
of elements of $[\lambda^+]^\lambda$
such that, for every cofinal $X\s\lambda^+$, there exists $\beta<\lambda^+$ such that $x_\beta\s X$; 
\item There exists a sequence $\langle x_\beta\mid \beta<\lambda^+\rangle$ 
of elements of $[\lambda^+]^\lambda$ satisfying the following.
For every sequence $\langle A_\alpha\mid \alpha<\lambda^+\rangle$ of elements of $[\lambda^+]^{\le\lambda}$ such that $|A_\alpha\cap A_\beta|<\lambda$ for all $\alpha<\beta<\lambda^+$,
for every cofinal $X\s\lambda^+$, there exists $\beta<\lambda^+$ such that $x_\beta\s X$ and,
for every $a\in[\lambda^+]^{<\cf(\lambda)}$, $|x_\beta\setminus\bigcup_{\alpha\in a}A_\alpha|=\lambda$.
\end{enumerate}
\end{fact}

\begin{theorem} Suppose that $\stick(\lambda^+)$ holds for an infinite regular cardinal $\lambda$. 

For every partition $\mathcal S$ of $E^{\lambda^+}_\lambda$ into stationary sets,
$\clubsuit_{\ad^*}(\mathcal S,\lambda,\lambda^+)$ holds.
\end{theorem}
\begin{proof} Let $\vec x=\langle x_\beta\mid\beta<\lambda^+\rangle$ be given by Fact~\ref{claim512}(2).
Fix a bijection $\pi:\lambda\leftrightarrow\lambda\times\lambda$ and then let $\pi_0,\pi_1$ be the unique maps from $\lambda$ to $\lambda$ 
to satisfy $\pi(j)=(\pi_0(j),\pi_1(j))$ for all $j<\lambda$.
For all nonzero $\alpha<\lambda^+$, fix a surjection $e_\alpha:\lambda\rightarrow\alpha$.

Define a sequence $\langle A_\alpha\mid\alpha<\lambda^+\rangle$ by recursion on $\alpha<\lambda^+$, as follows.
Set $A_\alpha:=\emptyset$.
Next, given a nonzero $\alpha<\lambda^+$ such that
$\langle A_{\bar\alpha}\mid\bar\alpha<\alpha\rangle$ has already been defined,
put $$J_\alpha:=\{ j<\lambda\mid |x_{e_\alpha(\pi_0(j))}\cap\alpha\setminus\bigcup\{ A_{e_\alpha(j')}\mid {j'\le j}\}|=\lambda\}.$$
Then, pick an injective sequence $\langle \xi_{\alpha,j}\mid j\in J_\alpha\rangle$ 
such that, for each $j\in J_\alpha$, $$\xi_{\alpha,j}\in x_{e_\alpha(\pi_0(j))}\cap\alpha\setminus\bigcup\{ A_{e_\alpha(j')}\mid {j'\le j}\}.$$
If $\{ \xi_{\alpha,j}\mid j\in J_\alpha\ \&\ \pi_1(j)=i\}$ happens to be cofinal in $\alpha$ for every $i<\lambda$, then we say that $\alpha$ is \emph{good}, and let
\begin{itemize}
\item $A_\alpha^i:=\{ \xi_{\alpha,j}\mid j\in J_\alpha\ \&\ \pi_1(j)=i\}$ for every $i<\lambda$,
and 
\item $A_\alpha:=\{\xi_{\alpha,j}\mid j\in J_\alpha\}$.
\end{itemize}
Otherwise, we just let $A_\alpha$ be any cofinal subset of $\alpha$ of order-type $\cf(\alpha)$,
and let $\langle A_\alpha^i\mid i<\cf(\alpha)\rangle$ be any partition of $A_\alpha$ into cofinal subsets of $\alpha$.

\begin{claim}\label{claim511} For all $\bar\alpha<\alpha<\lambda^+$, $|A_{\bar\alpha}\cap A_\alpha|<\lambda$.
\end{claim}
\begin{proof} If $\alpha$ is not good, then $\otp(A_\alpha)=\cf(\alpha)\le\lambda$, and the conclusion follows.
Next, suppose that $\alpha$ is good. Find $j'<\lambda$ such that $e_\alpha(j')=\bar\alpha$.
Then $A_{\bar\alpha}\cap A_\alpha\s\{ \xi_{\alpha,j}\mid j\in J_\alpha\cap j'\}$
\end{proof}

Next, given a cofinal $X\s\lambda^+$, 
for every $\epsilon<\lambda^+$, by Claim~\ref{claim511} and the choice of $\vec x$,
we may let $\beta_\epsilon$ denote the least $\beta<\lambda^+$ 
to satisfy both $x_{\beta}\s X\setminus\epsilon$ and
$|x_{\beta}\setminus\bigcup_{\alpha\in a}A_\alpha|=\lambda$ for every $a\in[\lambda^+]^{<\lambda}$.

Fix a set $E\in[\lambda^+]^{\lambda^+}$ on which the map $\epsilon\mapsto\beta_\epsilon$ is strictly increasing.
Consider the club $$D:=\{ \delta\in\acc^+(E)\mid \forall\epsilon\in E\cap\delta\, (\beta_\epsilon\cup x_{\beta_\epsilon}\s \delta)\}.$$

\begin{claim} Let $\delta\in D$. For every $i<\lambda$, $\sup(A_\delta^i\cap X)=\delta$.
\end{claim}
\begin{proof} Let $i<\lambda$ and let $\epsilon<\delta$. 
We shall show that there exists $j\in J_\delta$ such that 
$\xi_{\delta,j}$ is an element of $A_\delta^i\cap X\setminus\epsilon$.

Here we go.
By possibly increasing $\epsilon$, we may assume that $\epsilon\in E\cap\delta$.
Set $k:=e_\delta^{-1}(\beta_\epsilon)$,
and pick the unique $j<\lambda$ such that $\pi(j)=(k,i)$.
Then
$$x_{e_\delta(\pi_0(j))}\cap\delta\setminus\bigcup\{ A_{e_\delta(j')}\mid {j'\le j}\}=x_{\beta_\epsilon}\setminus \bigcup_{\alpha\in a}A_\alpha$$
for the set $a:=e_\delta[j+1]$ which is an element of $[\lambda^+]^{<\lambda}$. Consequently, $j\in J_\delta$,
and since $\pi_1(j)=i$,
$\xi_{\delta,j}$ is an element of $x_{\beta_\epsilon}\s X\setminus\epsilon$ that lies in $A_\delta^i$.
\end{proof}

It follows that for every partition $\mathcal S$ of $E^{\lambda^+}_\lambda$
into stationary sets, $\langle \{ A_\delta^i\mid i<\lambda\}\mid \delta\in E^{\lambda^+}_\lambda\rangle$
witnesses $\clubsuit_{\ad^*}(\mathcal S,\lambda,\lambda^+)$.
\end{proof}

\begin{lemma}\label{l51} Suppose that $\clubsuit_{\ad^*}(\{S\},\mu,\lambda^+)$ holds for some stationary subset $S$ of a successor cardinal $\lambda^+$.
Then $\clubsuit_{\ad^*}(\mathcal S,\mu,\lambda^+)$ holds for some partition $\mathcal S$ of $S$ into $\lambda^+$ many stationary sets.
\end{lemma}
\begin{proof} Let $\langle \{ A_\alpha^i\mid i<\mu\}\mid \alpha\in S\rangle$ be an array witnessing that $\clubsuit_{\ad^*}(\{S\},\mu,\lambda^+)$ holds.
Let $\mathcal I$ denote the collection of all $T\s S$ such that $\langle \{ A_\alpha^i\mid i<\mu\}\mid \alpha\in T\rangle$ fails to witness that $\clubsuit_{\ad^*}(\{T\},\mu,\lambda^+)$ holds.
It is not hard to see that $\mathcal I$ is a $\lambda^+$-complete proper ideal on $S$.
By Ulam's theorem, then, $\mathcal I$ is not weakly $\lambda^+$-saturated,
meaning that we may fix a partition $\mathcal S$ of $S$ into $\lambda^+$-many $\mathcal I^+$-sets.
Then $\langle \{A_\alpha^i\mid i<\mu\}\mid \alpha\in S\rangle$ witnesses that $\clubsuit_{\ad^*}(\mathcal S,\mu,\lambda^+)$ holds.
\end{proof}
\begin{defn}[\cite{MR2048518}] $\diamondsuit(\mathfrak b)$ asserts that for every function $F:{}^{<\omega_1}2\rightarrow{}^\omega\omega$,
there exists a function $g:\omega_1\rightarrow{}^{\omega}\omega$ with the property that for every function $f:\omega_1\rightarrow2$,
the set $\{ \alpha<\omega_1\mid F(f\restriction\alpha)\le^* g(\alpha)\}$ is stationary.
\end{defn}

\begin{cor} Suppose that $\diamondsuit(\mathfrak b)$ holds. 
Then:
\begin{enumerate}
\item $\clubsuit_{\ad^*}(\mathcal S,1,\omega_1)$ holds for some partition $\mathcal S$ of $\omega_1$ into uncountably many stationary sets;
\item There exist $2^{\aleph_1}$ many pairwise nonhomeomorphic Dowker spaces of size $\aleph_1$.
\end{enumerate}
\end{cor}
\begin{proof}  (1) By \cite[Theorem~5.5]{MR2048518}, $\diamondsuit(\mathfrak b)$ implies that 	$\clubsuit_{\ad^*}(\{\omega_1\},1,\omega_1)$ holds.
Now, the conclusion follows from Lemma~\ref{l51}.

(2) By Clause~(1) and Theorem~\ref{manyd} below.
\end{proof}

\renewcommand{\thesection}{A}
\section{Appendix: Many Dowker spaces}\label{many}

In this section, $\kappa$ denotes a regular uncountable cardinal.
By \cite[\S3]{paper48}, if $\clubsuit_{\ad}(\mathcal S,1,2)$ holds
for a partition $\mathcal S$ of some nonreflecting stationary subset of $\kappa$
into infinitely many stationary sets, then there exists a Dowker space of size $\kappa$.
Here, we demonstrate the advantage of $\mathcal S$ being large.
\begin{theorem}\label{manyd} Suppose that $\clubsuit_{\ad}(\mathcal S,1,2)$ holds,
where  $\mathcal S$ is a partition of a nonreflecting stationary subset of $\kappa$
into infinitely many stationary sets. Denote $\mu:=|\mathcal S|$.
Then there are $2^\mu$ many pairwise nonhomeomorphic Dowker spaces of size $\kappa$.
\end{theorem}
\begin{proof}
Fix an injective enumeration $\langle S^\zeta_{n}\mid \zeta<\mu, n<\omega\rangle$ of the elements of $\mathcal S$. 
As $\clubsuit_{\ad}(\mathcal S,1,2)$ holds, 
we may fix a sequence
$\langle A_{\alpha}\mid\alpha\in\bigcup\mathcal S\rangle$ such that:
\begin{itemize}
\item[(i)] For every $\alpha\in\bigcup\mathcal S$, $A_\alpha$ is a subset of $\alpha$, and for every $\alpha'\in\alpha\cap\bigcup\mathcal S$, $\sup(A_{\alpha'}\cap A_\alpha)<\alpha'$;
\item[(ii)] For all $B_0,B_1\in[\kappa]^{\kappa}$ and $(\zeta,n)\in\mu\times\omega$,
the following set is stationary:
$$G(S^\zeta_n,B_0,B_1):=\{\alpha\in S^\zeta_n\mid  \sup(A_\alpha\cap B_0)=\sup(A_\alpha\cap B_1)=\alpha\}.$$
\end{itemize}

For every nonempty $Z\s\mu$, we shall want to define a topological space $\mathbb X^Z$. 
To this end, fix a nonempty $Z\s\mu$.
For every $n<\omega$, let $S^Z_{n+1}:=\biguplus_{\zeta\in Z}S^\zeta_{n+1}$,
and then let $S^Z_0:=\kappa\setminus\biguplus_{n<\omega} S^Z_{n+1}$.
For every $\alpha<\kappa$, let $n^Z(\alpha)$ denote the unique $n<\omega$ such that $\alpha\in S^Z_n$.
For each $n<\omega$, let $ W^Z_n:=\bigcup_{i\leq n}S^Z_i $. 
Then, define a sequence $\vec {L^Z}=\langle L^Z_\alpha\mid\alpha<\kappa\rangle$ via:
$$L^Z_\alpha:=\begin{cases}
W^Z_{n^Z(\alpha)-1}\cap A_\alpha,&\text{if }n^Z(\alpha)>0\ \&\ \sup(W^Z_{n^Z(\alpha)-1}\cap A_\alpha)=\alpha;\\
\emptyset,&\text{otherwise.}
\end{cases}$$

Denote $S^Z:=\{ \alpha\in\acc(\kappa)\mid \sup(L^Z_\alpha)=\alpha\}$.
Finally, let $\mathbb X^Z=(\kappa,\tau^Z)$ 
be the ladder-system space determined by $\vec{L^Z}$, that is, a subset $U\s\kappa$ is $\tau^Z$-open iff, for every $\alpha\in U\cap S^Z$, $\sup(L^Z_\alpha\setminus U)<\alpha$.

\begin{claim}\label{lemma61} Let $Z$ and $Z'$ be nonempty subsets of $\mu$. Then:
\begin{enumerate}
\item For all $n<\omega$ and $\alpha\in S^Z_{n+1}$, $L^Z_\alpha\s W^Z_n$;
\item If $Z\setminus Z'$ is nonempty, then $S^Z\setminus S^{Z'}$ is stationary;
\item For all $\alpha\neq\alpha'$ from $S^Z$, $\sup(L^Z_\alpha\cap L^Z_{\alpha'})<\alpha$;
\item For all $B_0,B_1\in[\kappa]^\kappa$, there exists  $m<\omega$ such that, for every $n\in\omega\setminus m$, the following set is stationary:
\[ \{ \alpha\in S^Z_n \mid \sup(L^Z_\alpha\cap B_0)=\sup(L^Z_\alpha\cap B_1)=\alpha  \};\]
\item $S^Z$ is a nonreflecting stationary set.
\end{enumerate}
\end{claim}
\begin{proof} (1) Clear.

(2) Suppose that $Z\setminus Z'\neq\emptyset$, and pick $\zeta\in Z\setminus Z'$. 
As $W^Z_0=S^Z_0\supseteq S^\zeta_0$, the former is cofinal.
So, $S^Z\setminus S^{Z'}$ covers the stationary set $G(S^\zeta_1,W^Z_0,\kappa)$.

(3) For all $\alpha\neq\alpha'$ from $S^Z$, $\sup(L^Z_\alpha\cap L^Z_{\alpha'})\le\sup(A_\alpha\cap A_{\alpha'})<\alpha$.

(4) Pick $\zeta\in Z$. Given two cofinal subsets $B_0,B_1$ of $\kappa$, 
find $m_0,m_1<\omega$ be such that $|B_0\cap S^Z_{m_0}|=|B_1\cap S^Z_{m_1}|=\kappa$. Set $m:=\max\{m_0,m_1\}+1$. Then, for every $n\in\omega\setminus m$, 
$$G(S_n^\zeta,B_0\cap S^Z_{m_0},B_1\cap S^Z_{m_1})\s \{ \alpha\in S^Z_n \mid \sup(L^Z_\alpha\cap B_0)=\sup(L^Z_\alpha\cap B_1)=\alpha\}$$ and hence the latter is stationary.

(5) By Clause~(4), $S^Z$ is stationary. As $S^Z\s\bigcup_{n<\omega}S^Z_{n+1}\s\bigcup\mathcal S$,
and since $\bigcup\mathcal S$ is a nonreflecting stationary set, so is $S^Z$.
\end{proof}

By the preceding claim,
and the results of \cite[\S3]{paper48}, for every nonempty $Z\s\mu$, $\mathbb X^Z$ is a Dowker space. Thus we are left with proving the following:

\begin{claim} Suppose that $Z$ and $Z'$ are two distinct nonempty subsets of $\mu$.
Then $\mathbb X^Z$ and $\mathbb X^{Z'}$ are not homeomorphic.
\end{claim}
\begin{proof} Without loss of generality, we may pick $\zeta\in Z\setminus Z'$. Towards a contradiction,
suppose that $f:\kappa\leftrightarrow\kappa$ forms an homeomorphism from $\mathbb X^Z$ to $\mathbb X^{Z'}$.
As $f$ is a bijection, there are club many $\alpha<\kappa$ such that $f^{-1}[\alpha]=\alpha$.
By Claim~\ref{lemma61}(2), then, we may pick some $\alpha\in S^Z\setminus S^{Z'}$ such that $f^{-1}[\alpha]=\alpha$.
Set $\beta:=f(\alpha)$.

$\br$ If $\beta\notin S^{Z'}$, then $U:=\{\beta\}$ is a $\tau^{Z'}$-open neighborhood of $\beta$.

$\br$ If $\beta\in S^{Z'}$, then $\beta>\alpha+1$ and the ordinal interval $U:=[\alpha+1,\beta+1]$ is a $\tau^{Z'}$-open neighborhood of $\beta$.

In both cases, $U\s\kappa\setminus\alpha$, so that $f^{-1}[U]\s f^{-1}[\kappa\setminus\alpha]=\kappa\setminus\alpha$.
As $f$ is continuous and $U$ is a $\tau^{Z'}$-open neighborhood of $f(\alpha)$, $f^{-1}[U]$ must be a $\tau^Z$-open neighborhood of $\alpha$, contradicting the fact that $f^{-1}[U]$ is disjoint from $L^Z_\alpha$.
\end{proof}
This completes the proof.
\end{proof}
\section*{Acknowledgments}
The first author is partially supported by the European Research Council (grant agreement ERC-2018-StG 802756) and by the Israel Science Foundation (grant agreement 2066/18). 
The second author is supported by the European Research Council (grant agreement ERC-2018-StG 802756).

\end{document}